\documentclass{amsart}
\usepackage{amsmath,amssymb,amscd,latexsym}

\usepackage[all]{xy}

\usepackage{graphicx}

\usepackage{enumitem}

\usepackage[usenames, dvipsnames]{color}

\newcommand{\bi}{\begin{itemize}}
\newcommand{\ei}{\end{itemize}}
\newcommand{\bn}{\begin{enumerate}}
\newcommand{\en}{\end{enumerate}}

\newcommand{\bq}{\begin{equation}}
\newcommand{\eq}{\end{equation}}

\newcommand{\ba}{\begin{align}}
\newcommand{\ea}{\end{align}}

\newcommand{\bas}{\begin{align*}}
\newcommand{\eas}{\end{align*}}

\newcommand{\bs}{\begin{split}}
\newcommand{\es}{\end{split}}

\newcommand{\C}{{\mathbb{C}}}

\newcommand{\F}{{\mathbb{F}}}

\newcommand{\Pro}{{\mathbb{P}}}

\newcommand{\Z}{{\mathbb{Z}}}
\newcommand{\cl}{\mathrm{cl}}

\newcommand{\et}{\mathrm{\acute{e}t}}

\newcommand{\Hdg}{\mathrm{Hdg}}

\newcommand{\Mh}{{\mathcal M}}

\newtheorem{theorem}{Theorem}[section]
\newtheorem{lemma}[theorem]{Lemma}
\newtheorem{prop}[theorem]{Proposition}

\newtheorem*{theorem*}{Theorem}

\theoremstyle{definition}

\newtheorem{remark}[theorem]{Remark}

\begin{document}

\author{Gereon Quick}

\address{Department of Mathematical Sciences, NTNU, NO-7491 Trondheim, Norway}
\email{gereon.quick@math.ntnu.no}

\title{Examples of non-algebraic classes in the Brown-Peterson tower}

\date{}

\begin{abstract}
For every $n\ge 0$, we construct classes in the Brown-Peterson cohomology $BP\langle n \rangle$ of smooth projective complex algebraic varieties which are not in the image of the cycle map from the corresponding motivic Brown-Peterson cohomology. This generalizes the examples of Atiyah and Hirzebruch to all finite levels in the Brown-Peterson tower. 
\end{abstract}

\maketitle

\section{Introduction}

Let $X$ be a smooth projective complex algebraic variety. Let 
\[
\cl \colon CH^*(X) \to H^{2*}(X; \Z)
\]
be the cycle map from Chow groups to the singular cohomology of the space $X(\C)$ of complex points of $X$. Recall that classes in the image of $\cl$ are called {\it algebraic} and that all algebraic classes are contained in the subgroup of integral Hodge classes. However, it is well known that, in general, not all integral Hodge classes are algebraic. There are basically two types of examples of non-algebraic Hodge classes, one which can be detected by topological methods and one which cannot be detected by topological invariants. 
In \cite{ah}, Atiyah and Hirzebruch used the Atiyah-Hirzebruch spectral sequence to construct an obstruction for elements being in the image of $\cl$ and provided, for every prime $p$, examples of non-algebraic $p$-torsion classes using the Godeaux-Serre construction of varieties associated to finite groups. In \cite{totaro}, Totaro showed that the Atiyah-Hirzebruch obstruction can also be explained by the fact that $\cl$ factors through the natural map 
$\bar{\vartheta} \colon MU^{2*}(X)\otimes_{MU^*}\Z \to H^{2*}(X; \Z)$,  
where $MU^*(X)$ denotes the complex cobordism of the space $X(\C)$. A crucial fact is that the map $\bar{\vartheta}$ is neither surjective nor injective in general. 
In \cite{kollar}, Koll\'ar provided examples of varieties $X$ and non-torsion classes $\alpha \in H^4(X; \Z)$ which are not algebraic while a multiple of $\alpha$ is algebraic. In \cite{sv}, Soul\'e and Voisin explain these examples in detail and show that there is no locally constant invariant that can detect these examples. In particular, for those varieties, $\bar{\vartheta}$ is surjective. Moreover, Soul\'e and Voisin construct other types of non-algebraic torsion classes which cannot be explained by the obstruction of Atiyah-Hirzebruch and Totaro. We briefly discuss some more examples in Remark \ref{remtate}. 

One way to define the cycle map for smooth complex varieties is to interpret it as the natural map 
\begin{align*}
\cl \colon H_{\Mh}^{2*,*}(X; \Z) \to H^{2*}(X; \Z) 
\end{align*}
from motivic to singular cohomology induced by topological realization. The advantage of this definition for our purposes is that it immediately generalizes to other motivic cohomology theories. It is then a natural question what we can say about the image of the corresponding cycle map for other motivic theories.

The purpose of this paper is to show that the examples of Atiyah and Hirzebruch can be generalized to produce non-algebraic classes in all levels of the tower of Brown-Peterson cohomology theories which interpolate between singular cohomology and the $p$-localization of complex cobordism. 

To be more precise, for a prime $p$, let $BP^*(-)$ denote the Brown-Peterson cohomology for $p$ (see \cite{bp} and \cite{quillenfgl}). 
For an integer $n \ge 0$, let $BP\langle n \rangle$ be the associated intermediate theory studied in \cite{wilson}. 
These cohomology theories fit into a sequence
\begin{align*}
BP^*(X) \to \cdots \to BP \langle n \rangle^*(X) \to \cdots \to BP \langle 0 \rangle^*(X) = H^*(X;\Z_{(p)}).
\end{align*}

For every $p$ and $n$, there is a corresponding motivic Brown-Peterson cohomology which we denote by $BP\langle n \rangle_{\Mh}^{*,*}(X)$ (see \cite[\S 6.4]{hoyois} and \cite[\S 3]{ormsby}). For every $i$ and $j$, the topological realization functor from  the motivic to the classical stable homotopy category induces a natural homomorphism 
\begin{align*}
BP\langle n \rangle_{\Mh}^{i,j}(X) \to BP\langle n \rangle^{i}(X) 
\end{align*}
to the Brown-Peterson cohomology of the space $X(\C)$. 

For given $p$ and $n$, we write $w(n) := p^n + p^{n-1} + \cdots + p + 1$.
Our main result is the following. 

\begin{theorem*}\label{mainintro}
For every prime $p$ and every integer $n \ge 0$, there is a smooth projective complex algebraic variety $X$ and an element $b_n \in BP\langle n \rangle^{2w(n)+2}(X)$ which is not in the image of the map 
\begin{align*}
\cl_n \colon BP\langle n \rangle_{\Mh}^{2w(n)+2, w(n)+1}(X) \to BP\langle n \rangle^{2w(n)+2}(X).
\end{align*}
\end{theorem*}

The idea of the proof of the theorem and the structure of the paper can be summarized as follows. We start with the fundamental stable cofiber sequence
\begin{align*}
\Sigma^{2(p^n-1)}BP\langle n \rangle \to BP\langle n \rangle \to BP\langle n-1 \rangle. 
\end{align*}
It yields a well known obstruction for elements being in the image of the map $\rho^{n+1}_n \colon BP\langle n+1 \rangle^*(X) \to BP\langle n \rangle^*(X)$ provided by Milnor operations in mod $p$-cohomology $H^*(X; \F_p)$. However, it also provides a tool to lift elements from $H^*(X; \F_p)$ to $BP\langle n \rangle^{*+2w(n)-n-1}(X)$. 
We apply this observation to the elementary abelian $p$-group $G_{n+3}:=(\Z/p)^{n+3}$, for any prime number $p$, and construct explicit elements in $BP\langle n \rangle^*(BG_{n+3})$ which are not contained in the image of the map $\rho^{n+1}_n \colon BP\langle n+1 \rangle^*(BG_{n+3}) \to BP\langle n \rangle^*(BG_{n+3})$.  
Finally, we use these elements to prove the theorem for a Godeaux-Serre variety $X$ associated to $G_{n+3}$.  \\

Our argument relies on the fact that $\rho_n$ is not surjective in high degrees for the particular variety we consider. Wilson showed in \cite{wilson} that, for any finite complex $X$ and all $k\le 2w(n)$, the map $\rho_n \colon BP^k(X) \to BP\langle n \rangle^k(X)$ is surjective. The degree in which the examples of the theorem occur is hence minimal for our argument. 

In Remark \ref{remkollar}, we will briefly discuss a different type of argument for complex cobordism which the surjectivity of the map $MU^{2*}(X) \to H^{2*}(X;\Z)$ for Koll\'ar's examples.

{\bf Acknowledgements.} We are grateful to Claire Voisin for mentioning the argument of Remark \ref{remkollar} to us and for further very helpful comments. We would also like to thank Mike Hopkins for very helpful conversations.






\section{Obstructions and liftings}

Let $p$ be a prime number and $n \ge 0$ be an integer. 
Let $BP$ denote the spectrum representing Brown-Peterson cohomology for $p$ defined in \cite{bp}, and let $BP\langle n \rangle$ be the spectrum representing the associated intermediate theory for $p$ and $n$ studied by Wilson in \cite{wilson}. For $n=0$, one has $BP\langle 0 \rangle = H\Z_{(p)}$, the Eilenberg-MacLane spectrum for $\Z_{(p)}$, and,
for $n=-1$, we use the notation $BP\langle -1 \rangle := H\F_p$ for the mod $p$-Eilenberg-MacLane spectrum. For every $j > n \ge -1$, these theories are connected by canonical maps
\[
\rho_n \colon BP \to BP\langle n \rangle ~\text{and}~ \rho^j_n \colon BP\langle j \rangle \to BP\langle n \rangle.
\]
The coefficient rings are given by the polynomial algebras $BP^* \cong \Z_{(p)}[v_1, v_2, \ldots]$ and $BP\langle n \rangle^* \cong \Z_{(p)}[v_1, \ldots, v_n]$. The effect of the maps $\rho_n$ and $\rho^j_n$ is to send all $v_i$ with $i>n$ to $0$.

Recall that, for every $n\ge 1$, there is a stable cofiber sequence   
\begin{align*}
\Sigma^{2(p^n-1)}BP\langle n \rangle \xrightarrow{v_n} BP\langle n \rangle \xrightarrow{\rho^n_{n-1}} BP\langle n-1 \rangle \xrightarrow{q_n} \Sigma^{2(p^n-1)+1}BP\langle n \rangle.
\end{align*}

For every space $X$, this sequence induces a natural exact sequence
\begin{align*}
BP\langle n \rangle^{i+2(p^n-1)}(X) \xrightarrow{v_n} BP\langle n \rangle^i(X) \xrightarrow{\rho^n_{n-1}} BP\langle n-1 \rangle^i(X) \xrightarrow{q_n} BP\langle n \rangle^{i+2p^n-1}(X)
\end{align*}
where, by abuse of notation, we denote the induced maps on cohomology groups by the same symbols.

By \cite[Proposition 1.7]{wilson} (see also \cite[Lemma 2.4]{powell} and \cite[Proposition 4-4]{tamanoi00})), the map $q_n \colon BP\langle n-1 \rangle^i(X) \to BP\langle n \rangle^{i+2p^n-1}(X)$ corresponds, possibly up to a sign, to the $n$th Milnor operation $Q_n$ in mod $p$-cohomology in the sense that there is a commutative diagram
\begin{align}\label{1.4Qn}
\xymatrix{
BP\langle n \rangle^*(X) \ar[r]^-{\rho^{n}_{n-1}} \ar[dr]_-{\rho^{n}_{-1}} & BP\langle n-1 \rangle^i(X) \ar[d]_-{\rho^{n-1}_{-1}} \ar[r]^-{q_n} & BP\langle n \rangle^{i+2p^n-1}(X) \ar[d]^-{\rho^{n}_{-1}} \\
 & H^i(X; \F_p) \ar[r]_-{\pm Q_n} & H^{i+2p^n-1}(X; \F_p).}
\end{align}
Hence, since the top row of diagram \eqref{1.4Qn} is exact, $Q_n$ yields an obstruction to lifting an element from $BP\langle n-1 \rangle^*(X)$ to $BP\langle n \rangle^*(X)$ via $\rho^{n}_{n-1}$.

On the other hand, we can use the maps $q_n$ to produce explicit classes in $BP\langle n \rangle^*(X)$.  
Recall that the degree of $Q_i$ is $|Q_i| = 2p^i-1$ and hence 
\begin{align*}
\sum_{i=0}^{n}|Q_i| = \sum_{i=0}^{n}2p^i - 1 = 2w(n) - n - 1
\end{align*}
where we set $w(n):= p^{n}+\cdots +1$.
Successive composition of the maps $q_i$ yields a diagram which commutes (possibly up to sign) 
\begin{align*}
\scalebox{0.8}{
\xymatrix{
H\F_p \ar[r]^{q_0} \ar[d]_-{\pm Q_{n+1}\cdots Q_0} & \Sigma^{|Q_0|}BP\langle 0 \rangle \ar[r]^-{q_1} & \Sigma^{|Q_0|+|Q_1|}BP\langle 1 \rangle \ar[r]^-{q_1} & \cdots \ar[r]^-{q_{n}} & 
\Sigma^{2w(n)-n-1}BP\langle n \rangle \ar[d]^-{q_{n+1}} \\
\Sigma^{2w(n+1)-n-2}H\F_p & & & & \ar[llll]^-{\rho^{n+1}_{-1}} \Sigma^{2w(n+1)-n-2}BP\langle n+1 \rangle.}}
\end{align*}

Hence, for every $X$, there is a commutative diagram  
\begin{equation}\label{Qndiag}
\scalebox{0.85}{
\xymatrix{
BP\langle n+1 \rangle^{*+2w(n)-n-1}(X) \ar[r]^-{\rho^{n+1}_{n}} & BP\langle n \rangle^{*+2w(n)-n-1}(X) \ar[r]^-{q_{n+1}} & BP\langle n+1 \rangle^{*+2w(n+1)-n-2}(X) \ar[d]^{\rho^{n+1}_{-1}} \\
& H^*(X; \F_p) \ar[u]^{q_{n}\cdots q_0} \ar[r]_-{\pm Q_{n+1}\cdots Q_0} & H^{*+2w(n+1)-n-2}(X; \F_p)}}
\end{equation}
where the top row is exact. This yields the following criterion. 

\begin{lemma}\label{keylift}
If $x \in H^*(X; \F_p)$ satisfies $Q_{n+1}\cdots Q_0(x)\ne 0$, then 
\[
q_n \cdots q_0(x) \in BP\langle n \rangle^{*+2w(n)-n-1}(X)
\] 
is a non-trivial element which is not contained in the image of 
\begin{align*}
\rho^{n+1}_n \colon BP\langle n+1 \rangle^{*}(X) \to BP\langle n \rangle^{*}(X).
\end{align*}
\end{lemma}

\begin{remark}\label{wilsoninjective} 
In \cite{wilson}, Wilson showed that, for every space $X$ and integer $n \ge 0$, the natural homomorphism 
\[
BP^k(X) \to BP\langle n \rangle^k(X)
\]
is surjective for $k \le 2w(n)$. 
Hence the lowest even degree in which we can hope to find classes in $BP\langle n \rangle^k(X)$ which cannot be lifted to $BP\langle n+1 \rangle^k(X)$ is $k= 2w(n)+2$. This means that an element $x \in H^*(X; \F_p)$ with the properties in Lemma \ref{keylift} must be of degree at least $n+3$. 
\end{remark}

\section{$BP\langle n \rangle$-classes for elementary abelian $p$-groups}

In this section, we look at the $BP\langle n \rangle$-cohomology of the classifying spaces of elementary abelian $p$-groups. 
The generalized cohomology of such spaces is a well studied subject in the literature (see e.g. \cite{powell} and \cite{strickland} for the case $BP\langle n \rangle$). The goal of this section is merely to specify concrete elements with the properties needed to apply Lemma \ref{keylift}.  
The cases $p=2$ or $p$ odd are very similar. However, for the convenience of the reader, we provide the computations for both cases separately. 
To simplify the notation, for $G$ a group and $h$ a cohomology theory, we write $h^*(G)$ for $h^*(BG_+)$, where $BG$ denotes the classifying space of $G$.

\subsection{Non-liftable $BP\langle n \rangle$-classes for $p=2$}

We start with the case $p=2$. 
Recall that Milnor's operations 
\[
Q_n \colon H^i(X; \F_2) \to H^{i+2^{n+1}-1}(X; \F_2)
\]
are defined inductively in terms of Steenrod squares by 
\begin{align*}
Q_0 & = Sq^1,\\
Q_{n+1} & = Sq^{2^{n+1}}Q_n + Q_nSq^{2^{n+1}}.
\end{align*}

A very convenient fact about the $Q_n$'s is that they are primitive elements of the Steenrod algebra, i.e.,
\bq\label{Qprim}
Q_n(xy) = Q_n(x)y + xQ_n(y).
\eq

Let $G_k$ be the $k$-fold product of $\Z/2$, i.e., $G_k=(\Z/2)^k$. The $\F_2$-cohomology of $G_k$ is given by the formula
\[
H^*((\Z/2)^k; \F_2) \cong \F_2[x_1,\ldots,x_k].
\]


\begin{lemma}\label{1Qaction}
Let $x$ be a polynomial generator of $H^*((\Z/2)^k; \F_2)$. For $i\ge 0$, the Milnor operation $Q_i$ acts on $x$ by
\begin{align}
Q_i(x) & = x^{2^{i+1}} \label{Qix1} \\ 
Q_i(x^{2n}) & = 0 ~\text{for all}~n \ge 1 \label{Qixn}.
\end{align}
\end{lemma}
\begin{proof}
For $i=0$, we have $Q_0(x)=Sq^1(x)=x^2$ and 
\[
Q_0(x^{2n})= Sq^1x^{2n} = \binom {2n} {1} x^{2n+1} = 2nx^{2n+1} = 0
\]
for $n\ge 1$, since we are working modulo $2$. 
For $i\ge 1$, we proceed by induction using \eqref{Qprim}. For $n=0$, we get
\begin{align*}
Q_i(x) = (Sq^{2^i}Q_{i-1} + Q_{i-1}Sq^{2^i})(x) = Sq^{2^i}(x^{2i}) + Q_{i-1}(0)  = x^{2^{i+1}}.
\end{align*}
For $n\ge 1$, we have
\begin{align*}
Q_i(x^{2n}) = (Sq^{2^i}Q_{i-1} + Q_{i-1}Sq^{2^i})(x^{2n})  = Sq^{2^i}Q_{i-1}(x^{2n}) + Q_{i-1}Sq^{2^i}(x^{2n}) =0 
\end{align*}
by the induction hypothesis on $Q_{i-1}$. 
%
\end{proof}

\begin{lemma}\label{nQaction}
For $m\le k$, let $x_1,\ldots, x_m$ be distinct polynomial generators of $H^*((\Z/2)^k; \F_2)$. The effect of the iterated Milnor operations on the product $x_1\cdots x_m$ is given by
\begin{align}
Q_nQ_{n-1}\cdots Q_1Q_0(x_1\cdots x_m) & = 0 ~\text{for}~ m\le n \label{nQnx} \\
Q_nQ_{n-1}\cdots Q_1Q_0(x_1\cdots x_m) & = \sum x_1^{j_1}x_2^{j_2} \cdots x_m^{j_m} ~\text{for}~ m \ge n+1 \label{nQmx}
\end{align}
where the sum is taken over all permutations $(j_1, \ldots, j_m)$ of the set of $m$ numbers $\{2^{n+1}, 2^n, \ldots, 2^1, 1, \ldots, 1\}$ with $m-(n+1)$ many $1$'s. 
In particular, we get 
\begin{align*}
Q_nQ_{n-1}\cdots Q_1Q_0(x_1\cdots x_m) \ne 0 ~\text{for}~ m \ge n+1.
\end{align*}
\end{lemma}
\begin{proof}
For $n=0$, Lemma \ref{1Qaction} and equation \eqref{Qprim} imply 
\[
Q_0(x_1\cdots x_m) = x_1^2x_2\cdots x_m + x_1x_2^2\cdots x_m + \cdots + x_1x_2\cdots x_m^2
\]
where in each summand there is exactly one exponent equal to $2$ and all others are equal to $1$. This proves the case $n=0$. For $n=1$ and $m=1$, we have 
\[
Q_1Q_0(x_1)= Q_1(x_1^2)=0
\]
by \eqref{Qixn}. Hence equations \eqref{nQnx} and \eqref{nQmx} hold for $n=0$ and the case $n=1$ and $m=1$. 
For $n\ge 1$ and $m\ge 1$, we proceed by induction. 
If $m \le n-1$, then $Q_{n-1}\cdots Q_1Q_0(x_1\cdots x_m) = 0$ by the induction hypothesis. 
So assume $m \ge n$. By the induction hypothesis, we have 
\begin{align*}
Q_nQ_{n-1}\cdots Q_1Q_0(x_1\cdots x_m) & = Q_n(\sum x_1^{j_1}x_2^{j_2} \cdots x_m^{j_m})
\end{align*}
where the sum is taken over all permutations $(j_1, \ldots, j_m)$ of the set of $m$ numbers $\{2^{n}, 2^{n-1}, \ldots, 2^1, 1, \ldots, 1\}$. 
For each summand, equation \eqref{Qprim} implies
\begin{align}
 & Q_n(x_1^{j_1}x_2^{j_2} \cdots x_m^{j_m}) = Q_n(x_1^{j_1})(x_2^{j_2} \cdots x_m^{j_m}) + x_1^{j_1}Q_n(x_2^{j_2} \cdots x_m^{j_m)}) \nonumber \\
& = Q_n(x_1^{j_1})(x_2^{j_2} \cdots x_m^{j_m}) + x_1^{j_1}(Q_n(x_2^{j_2})(x_3^{j_3} \cdots x_m^{j_m}) + x_2^{j_2}Q_n(x_3^{j_3} \cdots x_m^{j_m})) \nonumber \\
& \vdots  \nonumber \\
& = Q_n(x_1^{j_1})x_2^{j_2} \cdots x_m^{j_m} + x_1^{j_1}Q_n(x_2^{j_2})x_3^{j_3} \cdots x_m^{j_m} + \cdots + x_1^{j_1} \cdots x_{m-1}^{j_{m-1}}Q_n(x_m^{j_m}). \label{Qnline}
\end{align}

If $m=n$, then each $(j_1, \ldots, j_n)$ is a bijection of $\{1, \ldots,m\}$ with $\{2^n, 2^{n-1}, \ldots, 2^2, 2^1\}$. In particular, we have $j_i \ge 2$ for all $i=1, \ldots, n$. By formula \eqref{Qixn}, this implies that $Q_n(x_i^{j_i}) =0$ for $i=1,\ldots, n$. Hence, if $m=n$, line \eqref{Qnline} is equal to zero. This finishes the proof of equation \eqref{nQnx}. 

If $m\ge n+1$, then we have $Q_n(x_i^{j_i}) =0$ if $j_i\ge 2$ by \eqref{Qixn} and $Q_n(x_i^{j_i}) = x_i^{2^{n+1}}$ if $j_i =1$ by \eqref{Qix1}. 
This implies that, for a fixed permutation $(j_1, \ldots, j_m)$ of the set of $m$ numbers $\{2^{n}, 2^{n-1}, \ldots, 2^1, 1, \ldots, 1\}$, 
the element $Q_n(x_1^{j_1}x_2^{j_2} \cdots x_m^{j_m})$ is the sum of terms of the form $x_1^{j_1}\cdots x_{i-1}^{j_{i-1}} x_i^{2^{n+1}} x_{i+1}^{j_{i+1}} \cdots x_m^{j_m}$, one summand for each $i$ with $j_i=1$. 
Taking the sum over all permutations $(j_1, \ldots, j_m)$ of the set of $m$ numbers $\{2^{n}, 2^{n-1}, \ldots, 2^1, 1, \ldots, 1\}$ ($m-n$ many $1$'s), then yields
\begin{align*}
Q_nQ_{n-1}\cdots Q_1Q_0(x_1\cdots x_m) & = \sum x_1^{j'_1}x_2^{j'_2} \cdots x_m^{j'_m}
\end{align*}
where the sum is taken over all bijections $(j'_1, \ldots, j'_m)$ of $\{1, \ldots, m\}$ with the set of $m$ numbers $\{2^{n+1}, 2^n, \ldots, 2^1, 1, \ldots, 1\}$ ($m-n-1$ many $1$'s). 
This finishes the proof of equation \eqref{nQmx} and the lemma. 
\end{proof}

Recall that the degree of $Q_i$ is $|Q_i| = 2^{i+1}-1$ and hence 
\begin{align*}
\sum_{i=0}^{n}|Q_i| = \sum_{i=0}^{n}2^{i+1} - 1 = 2^{n+2} - 2 - (n+1) = 2^{n+2}-3-n.
\end{align*}

\begin{prop}\label{p=2nt}
Let $n$ be an integer $\ge 0$ and let $k$ and $m$ be integers such that $k\ge m \ge n+2$. Let $G_k = (\Z/2)^k$ and $x_1,\ldots, x_m$ be distinct polynomial generators of $H^*((\Z/2)^k; \F_2)$. 
Then the element $q_{n}\cdots q_0(x_1\cdots x_m)$ is nontrivial in the group $BP\langle n \rangle^{m+2^{n+2}-3-n}(G_k)$ and is not contained in the image of the map   
\[
\rho^{n+1}_n \colon BP\langle n+1 \rangle^{m+2^{n+2}-3-n}(G_k) \to BP\langle n \rangle^{m+2^{n+2}-3-n}(G_k).
\]
\end{prop}
\begin{proof}
We know by diagram \eqref{Qndiag} that 
\[
\rho^{n+1}_{-1}q_{n+1}(q_{n}\cdots q_0(x_1\cdots x_m)) = Q_{n+1}\cdots Q_0(x_1\cdots x_m) ~\text{in}~H^{m+2^{n+3}-4-n}(G_k; \F_2).
\] 
By Lemma \ref{nQaction}, we know that $Q_{n+1}\cdots Q_0(x_1\cdots x_m)$ nontrivial if $m\ge n+2$. 
The assertion then follows from Lemma \ref{keylift}.  
\end{proof}

\begin{remark}\label{p=2interestingcase}
For our application to algebraic varieties, we are interested in elements in even degree in $BP\langle n \rangle^*(G_k)$. Hence the minimal values for $k$ and $m$ such that $q_{n}\cdots q_0(x_1\cdots x_m)$ is nontrivial is $k=m=n+3$. In this case we have 
\[
q_{n}\cdots q_0(x_1\cdots x_{n+3}) \ne 0 ~\text{in}~ BP\langle n \rangle^{2^{n+2}}(G_{n+3})
\]
and cannot be lifted to $BP\langle n+1 \rangle^{2^{n+2}}(G_{n+3})$. 
\end{remark}


\subsection{Non-liftable $BP\langle n \rangle$-classes for odd primes}

In this section, let $p$ be an odd prime. 
Let $G_k$ be the $k$-fold product of $\Z/p$, i.e., $G_k=(\Z/p)^k$. The $\F_p$-cohomology of $G_k$ is given by 
\begin{align}\label{gcohomology}
H^*((\Z/p)^k; \F_p) \cong \Lambda(x_1,\ldots, x_k)\otimes \F_p[y_1,\ldots, y_k]
\end{align}
with $|x_i|=1$ and $|y_i|=2$ (and $x_i^2=0$) for $i=1, \ldots, k$. 
The Bockstein homomorphism 
\[
\beta \colon H^*(G_k;\Z/p) \to H^{*+1}(G_k; \Z/p).
\]
sends $x_i$ to $y_i$, i.e., $\beta(x_i)=y_i$ for $i=1, \ldots, k$. The Milnor operations $Q_n$ in the mod $p$-Steenrod algebra can be defined recursively by
\begin{align*}
Q_0 & = \beta,\\
Q_{n+1} & = P^{p^n}Q_n - Q_nP^{p^n}
\end{align*}
where $P^i$ is the $i$th reduced $p$th power operation. The $Q_n$ are primitive elements, i.e., $Q_n(xy) = Q_n(x)y + (-1)^{|x|\cdot |Q_n|}xQ_n(y)$. Since the degree of $Q_n$ is always odd, this means
\bq\label{pQprim}
Q_n(xy) = Q_n(x)y + (-1)^{|x|}xQ_n(y).
\eq

\begin{lemma}
The action of the Milnor operations on the generators of $H^*((\Z/p)^k; \F_p)$ is given by 
\begin{align}
Q_n(x_i) & = y_i^{p^n} \label{Qnxp} \\
Q_n(y_i) & = 0. \label{Qnyp}
\end{align}
\end{lemma}
\begin{proof}
For $n=0$, we have $Q_0(x_i)=\beta(x_i)=y_i$, and $Q_0(y_i)=\beta(\beta(x_i))=0$. 
For $n\ge 1$, we proceed by induction. For $x_i$, we get
\begin{align*}
Q_n(x_i) & = (P^{p^{n-1}}Q_{n-1} - Q_{n-1}P^{p^{n-1}})(x_i) \\
& = P^{p^{n-1}}Q_{n-1}(x_i) - Q_{n-1}P^{p^{n-1}}(x_i) \\
& = P^{p^{n-1}}(y_i^{p^{n-1}}) - Q_{n-1}(0) = (y_i^{p^{n-1}})^p \\
& = y_i^{p^n}.
\end{align*}
For $y_i$, we get 
\begin{align*}
Q_n(y_i) & = (P^{p^{n-1}}Q_{n-1} - Q_{n-1}P^{p^{n-1}})(y_i) \\
& = P^{p^{n-1}}Q_{n-1}(y_i) - Q_{n-1}P^{p^{n-1}})(y_i)\\
& = 0
\end{align*}
since $Q_{n-1}$ acts trivially on $y_i$ by the induction hypothesis (and $P^{p^{n-1}}(y_i)=0$ for $n\ge 2$). 
\end{proof}

In order to facilitate the bookkeeping in the next lemma, we set 
\begin{align}\label{halfnot}
y_i^{\frac{1}{2}}:=x_i~\text{for}~i=1,\ldots,k
\end{align}
in $H^*((\Z/p)^k; \F_p) = \Lambda(x_1,\ldots, x_k)\otimes \F_p[y_1,\ldots, y_k]$.

\begin{lemma}\label{nQactionp}
With notation as in \eqref{gcohomology} and \eqref{halfnot} and $m\le k$, we have 
\begin{align*}
Q_nQ_{n-1}\cdots Q_0(x_1\cdots x_m) & = 0 ~\text{for}~ m\le n \\ 
Q_nQ_{n-1}\cdots Q_0(x_1\cdots x_m) & = \sum (-1)^{\rho(j)} y_1^{j_1}y_2^{j_2} \cdots y_m^{j_m} ~\text{for}~ m \ge n+1 
\end{align*}
where the sum is taken over all bijections $j=(j_1, \ldots, j_m)$ of the set $\{1, \ldots, m\}$ with the set of $m$ numbers $\{p^n, p^{n-1}, \ldots, p, 1, \frac{1}{2}, \ldots, \frac{1}{2}\}$ with $m-(n+1)$ many $\frac{1}{2}$'s, and $\rho(j)$ is the sum over the numbers $a_t$, for $t\in \{1, \ldots, n\}$, defined as follows: let $i(t)$ be the index with $j_{i(t)}=p^t$; then $a_t$ is the number of indices $s \in \{1, \ldots, m\}$ with $s < i(t)$ and either $j_s=\frac{1}{2}$ or $j_s >p^t$. 

In particular, we get 
\begin{align*}
Q_nQ_{n-1}\cdots Q_1Q_0(x_1\cdots x_m) \ne 0 ~\text{for}~ m\ge n+1.
\end{align*}
\end{lemma}
\begin{proof}
For $n=0$, we have 
\[
Q_0(x_1\cdots x_m) = y_1x_2\cdots x_m - x_1y_2x_3\cdots x_m + \cdots + (-1)^m x_1x_2\cdots y_m.
\]
Keeping notation \eqref{halfnot} in mind, this proves the assertion for $n=0$.

For $n=1$ and $m=1$, we have 
$Q_1Q_0(x_1)= Q_1(y_1)=0$ 
by \eqref{Qnyp}. 
For $n\ge 1$ and $m\ge 1$, we proceed by induction. 
If $m \le n-1$, then $Q_{n-1}\cdots Q_1Q_0(x_1\cdots x_m) = 0$ by the induction hypothesis. 
So assume $m \ge n$. By the induction hypothesis, we have 
\begin{align*}
Q_nQ_{n-1}\cdots Q_0(x_1\cdots x_m) & = Q_n(\sum (-1)^{\rho(j)} y_1^{j_1}y_2^{j_2} \cdots y_m^{j_m})
\end{align*}
where the sum is taken over all bijections $j=(j_1, \ldots, j_m)$ of the set $\{1, \ldots, m\}$ with the set of $m$ numbers $\{p^{n-1}, \ldots, p, 1, \frac{1}{2}, \ldots, \frac{1}{2}\}$ (with $m-n$ many $\frac{1}{2}$'s).

If $m=n$, then each $(j_1, \ldots, j_n)$ is a permutation of the set $\{p^{n-1}, \ldots, p, 1\}$, or in other words, there is no $x_i$ left. Then \eqref{Qnyp} implies that all summands vanish under $Q_n$. This proves the assertion for $m=n$.

Now we assume $m\ge n+1$. Let $j$ be a fixed bijection from $\{1, \ldots, m\}$ to $\{p^{n-1}, \ldots, p, 1, \frac{1}{2}, \ldots, \frac{1}{2}\}$. 
By formula \eqref{pQprim}, applying $Q_n$ to $y_1^{j_1}y_2^{j_2} \cdots y_m^{j_m}$ yields new summands, one for each $i\in \{1, \ldots, m\}$ with $j_i=\frac{1}{2}$, of the form
\[
(-1)^{a} y_1^{j_1}\cdots y_{i-1}^{j_{i-1}}y_i^{p^n}y_{i+1}^{j_{i+1}} \cdots y_m^{j_m}
\]
where $a$ is given by the number of indices $s \in \{1, \ldots, m\}$ with $s<i$ and $j_s=\frac{1}{2}$. Now taking the sum over all bijections $j$ from $\{1, \ldots, m\}$ to $\{p^{n-1}, \ldots, p, 1, \frac{1}{2}, \ldots, \frac{1}{2}\}$, we obtain that 
\[
Q_n\cdots Q_0(x_1\cdots x_m) = \sum (-1)^{\rho(j')} y_1^{j'_1}y_2^{j'_2} \cdots y_m^{j'_m}
\]
where the sum is now over all bijections $j'$ of $\{1, \ldots, m\}$ with $\{p^n, \ldots, p, 1, \frac{1}{2}, \ldots, \frac{1}{2}\}$. 
\end{proof}

Recall $|Q_i| = 2p^i-1$ and $\sum_{i=0}^{n}|Q_i| = \sum_{i=0}^{n}2p^i - 1 = 2w(n) - n-1$ where we write $w(n):= p^n+\cdots +1$.

\begin{prop}\label{pnt}
Let $n$ be an integer $\ge 0$, and let $k$ and $m$ be integers such that $k\ge m \ge n+2$. Let $G_k = (\Z/p)^k$ and $x_1,\ldots, x_m$ be distinct exterior algebra generators of $H^*((\Z/p)^k; \F_p)$ as in formula \eqref{gcohomology}. 
Then the element $q_{n}\cdots q_0(x_1\cdots x_m)$ is nontrivial in the group $BP\langle n \rangle^{m+2w(n)-n-1}(G_k)$ and is not contained in the image of the map 
\[
\rho^{n+1}_n \colon BP\langle n+1 \rangle^{m+2w(n)-n-1}(G_k) \to BP\langle n \rangle^{m+2w(n)-n-1}(G_k).
\]
\end{prop}
\begin{proof}
By diagram \eqref{Qndiag}, we know  
\[
\rho^{n+1}_{-1}q_{n+1}(q_{n}\cdots q_0(x_1\cdots x_m)) = \pm Q_{n+1}\cdots Q_0(x_1\cdots x_m)  \in H^{m+2w(n+1)-n-2}(G_k; \F_p).
\] 
By Lemma \ref{nQactionp}, $Q_{n+1}\cdots Q_0(x_1\cdots x_m)$ is nontrivial if $m\ge n+2$. 
The assertion then follows from Lemma \ref{keylift}. 
\end{proof}


By \cite[Corollary 7.10]{powell}, we know that any element in $BP\langle n \rangle^*(G_k)$ which is not in the image of the map 
$BP^*(G_k) \to BP\langle n \rangle^*(G_k)$ 
is in the image of the map 
\[
q_{n}\cdots q_0 \colon H^*(G_k; \F_p) \to BP\langle n \rangle^{*+2w(n)-n-1}(G_k).
\]
The point of Propositions \ref{p=2nt} and \ref{pnt} is that we specify concrete nontrivial elements in this image that we can use for our application in the next section. 


\section{Non-algebraic classes in $BP\langle n \rangle$-cohomology}

Let $X$ be a smooth projective complex algebraic variety, and let $H_{\Mh}^{i,j}(X; R)$ denote its motivic cohomology with coefficients in a ring $R$. 
For every $i$ and $j$, the topological realization functor $X \mapsto X(\C)$ induces a natural homomorphism 
\begin{align}\label{Hmap}
H_{\Mh}^{i,j}(X; R) \to H^i(X; R) 
\end{align}
to the singular cohomology of the space $X(\C)$ of complex points of $X$.

For a prime $p$ and integer $n\ge 0$, let $BP\langle n \rangle_{\Mh}^{i,j}(X)$ be the motivic Brown-Peterson cohomology for $p$ and $n$ constructed in \cite[\S 6.4]{hoyois} and \cite[\S 3]{ormsby}. Again, for every $i$ and $j$, the topological realization functor from the motivic to the classical stable homotopy category induces a natural homomorphism 
\begin{align}\label{BPnmap}
BP\langle n \rangle_{\Mh}^{i,j}(X) \to BP\langle n \rangle^{i}(X) 
\end{align}
to the Brown-Peterson cohomology for $p$ and $n$ of the space $X(\C)$.

Recall that in degrees $(2i,i)$, the group $H_{\Mh}^{2i,i}(X; \Z)$ is naturally isomorphic to the Chow group $CH^i(X)$ of codimension $i$ cycles on $X$. Therefore, we will denote the map \eqref{BPnmap} in degrees $(2i,i)$ by 
\[
\cl_n \colon BP\langle n \rangle_{\Mh}^{2i,i}(X) \to BP\langle n \rangle^{2i}(X)
\]
and extend this notation to the map \eqref{Hmap} for mod $p$-cohomology  
\begin{align*}
\cl_{-1} \colon H_{\Mh}^{2i,i}(X; \F_p) \to H^{2i}(X; \F_p).
\end{align*}

For every $n$ and $i$, these maps fit into a commutative diagram 
\begin{align*}
\xymatrix{
BP\langle n \rangle_{\Mh}^{2i,i}(X) \ar[d]_-{\rho^n_{-1,\Mh}} \ar[r]^-{\cl_n} & BP\langle n \rangle^{2i}(X) \ar[d]^-{\rho^n_{-1}} \\
H_{\Mh}^{2i,i}(X; \F_p) \ar[r]_-{\cl_{-1}} & H^{2i}(X; \F_p).}
\end{align*}


\begin{lemma}\label{ltm}
Let $n \ge 0$ and $X$ be a smooth projective complex variety. Let $b_n$ be an element in $BP\langle n \rangle^{*}(X)$ which is not contained in the image of 
\[
\rho^{n+1}_n \colon BP\langle n+1 \rangle^{*}(X) \to BP\langle n \rangle^{*}(X)
\]
and has nontrivial image under $\rho^n_{-1}$ in $H^*(X; \F_p)$. 
Then $b_n$ is not contained in the image of $\cl_n$.  
\end{lemma}
\begin{proof}
Let $y_n \in H^*(X; \F_p)$ be the image of $b_n$ under $\rho^n_{-1}$. If $b_n$ was in the image of $\cl_n$, then $y_n$ was in the image of $\cl_{-1}$ as well. But, by the work of Totaro \cite[Theorem 3.1]{totaro} (see also Levine-Morel \cite[Theorem 1.2.19]{lm}),  
the mod $p$-cycle map $\cl_{-1}$ factors through the natural map
\begin{align*}
BP^*(X)\otimes_{BP^*}\Z/p \to H^*(X; \F_p)
\end{align*}
induced by $\rho_{-1} \colon BP^*(X) \to H^*(X; \F_p)$. Hence $y_n$ would have to be contained in the image of $\rho_{-1}$.  
But this is impossible by our assumption that $b_n$ is not contained in the image of $\rho^{n+1}_n$ and hence not in the image of 
\[
\rho_n \colon BP^*(X) \to BP\langle n \rangle^{*}(X).
\]
\end{proof}

We can now extend the argument of Atiyah-Hirzebruch to all finite levels in the Brown-Peterson tower and prove our main result. 

\begin{theorem}\label{main}
For every prime $p$ and every integer $n \ge 0$, there is a smooth projective complex algebraic variety $X$ and an element $b_n$ in $BP\langle n \rangle^{2w(n)+2}(X)$ which is not contained in the image of the natural map 
\begin{align*}
\cl_n \colon BP\langle n \rangle_{\Mh}^{2w(n)+2, w(n)+1}(X) \to BP\langle n \rangle^{2w(n)+2}(X).
\end{align*}
\end{theorem}
\begin{proof}
Let $G_{n+3} = (\Z/p)^{n+3}$. Let $k$ denote $2w(n+1)+1$.   
By \cite[\S 20]{serre} and \cite[Proposition 6.6]{ah}, there is a smooth projective variety $X$ of complex dimension $k$ together with a continuous map $X \to BG_{n+3} \times K(\Z, 2)$ which is $k$-connected. 
Let 
\[
\varphi \colon X \to BG_{n+3} \times K(\Z, 2) \to BG_{n+3}
\] 
denote the composition with the projection onto $BG_{n+3}$. For all $i \le k$, any nonzero element $x\in H^i(G_{n+3}; \F_p)$ is pulled back to a nonzero element $\varphi^*(x) \in H^i(X; \F_p)$.  

Let $x_1,\ldots, x_{n+3}$ denote the distinct exterior algebra generators of $H^*(G_{n+3}; \F_p)$ for $p$ odd, or the distinct polynomial generators for $p=2$. By Lemmas \ref{nQaction} and \ref{nQactionp}, we know that the element $Q_n\cdots Q_0(x_1 \cdots x_{n+3})$ is nontrivial in $H^{2w(n)+2}(G_{n+3};\F_p)$. Thus, since $2w(n)+2 \le k$, the element 
\[
y_n:= Q_n\cdots Q_0(\varphi^*(x_1 \cdots x_{n+3})) = \varphi^*Q_n\cdots Q_0(x_1 \cdots x_{n+3}) \in H^{2w(n)+2}(X;\F_p)
\]
is nontrivial as well. 
Now we define 
\[
b_n:= q_n\cdots q_0(\varphi^*(x_1\cdots x_{n+3})) = \varphi^*q_n\cdots q_0(x_1\cdots x_{n+3}) \in BP\langle n \rangle^{2w(n)+2}(X).
\]
Since its image under the canonical map
\[
BP\langle n \rangle^{2w(n)+2}(X) \to H^{2w(n)+2}(X; \F_p)
\]
is $\pm y_n$, $b_n$ is nontrivial as well. 
Hence, by Lemma \ref{keylift}, $b_n$ is not contained in the image of $\rho^{n+1}_n \colon BP\langle n+1 \rangle^{*}(X) \to BP\langle n \rangle^{*}(X)$. 
By Lemma \ref{ltm} this implies that $b_n$ is not contained in the image of $\cl_n$. 
%
\end{proof}

\begin{remark}
The minimal possible (complex) dimension of the variety in Theorem \ref{main} is 
$2w(n+1)+1 = 2(p^{n+1} + \cdots + 1) +1$. 
This is because we need the map $X \to BG_{n+3} \times K(\Z, 2)$ to be $2w(n+1)+1$-connected for the argument to work. Hence, for $p=2$, we see that the least possible dimension of $X$ is $2^{n+3}-1$. 
The fact that the dimension is rather big is consistent with a result of Soul\'e-Voisin \cite[Theorem 1]{sv} that the order of an element detected by the Atiyah-Hirzebruch-Totaro obstruction must be small relative to the dimension of the variety.   
\end{remark}

\begin{remark}
The case $n=0$ of Theorem \ref{main} is the original example of Atiyah and Hirzebruch \cite[p. 42, Proof of (6.7)]{ah}. Let $G_3=\Z/p\times \Z/p\times \Z/p$ with cohomology ring $H^*(G_3; \F_p) = \Lambda(x_1, x_2, x_3)\otimes \F_p[y_1, y_2, y_3]$. 
Atiyah and Hirzebruch consider the element $y:=\beta(x_1x_2x_3) \in H^4(G_3; \F_p)$, where $\beta$ denotes the Bockstein homomorphism, and show $Q_1(y) \ne 0$ in $H^{4+2p-1}(G_3; \F_p)$. Since $y$ is in the kernel of $\beta$, it corresponds to a unique $p$-torsion element in $H^*(G_3; \Z)$. The Godeaux-Serre construction then provides a smooth projective variety $X$ whose cohomology contains $H^*(G_3; \Z)$ as a direct factor up to degree $\dim_{\C}X$. Then they show that a class in the image of $\cl$ must be a permanent cycle in the Atiyah-Hirzebruch spectral sequence $H^*(X; \Z) \Rightarrow KU^*(X)$ converging to the complex $K$-theory of $X$. The operation $Q_1$ corresponds to a differential in this spectral sequence. Hence if $Q_1$ acts non-trivially on $y$, then the lift of $y$ to $H^*(X; \Z)$ cannot be a permanent cycle. 
\end{remark}

\begin{remark}
The elements constructed in Theorem \ref{main} map to torsion classes in the $\Z_{(p)}$-cohomology of $X(\C)$. In particular, their images in $H^*(X; \Z_{(p)})$ are $\Z_{(p)}$-Hodge classes. 
\end{remark}

\begin{remark}\label{remkollar}
Let $\Omega^*(-)$ denote the algebraic cobordism theory of Levine and Morel \cite{lm}. It is constructed as the universal oriented cohomology theory for smooth algebraic varieties over fields of characteristic $0$. For smooth varieties over $\C$, it comes equipped with a natural commutative diagram 
\begin{align}\label{cyclediag}
\xymatrix{
\Omega^*(X) \ar[d]_-{\theta} \ar[r]^-{\cl_{\Omega}} & MU^{2*}(X) \ar[d]^-{\vartheta} \\
CH^*(X) \ar[r]_-{\cl} & H^{2*}(X; \Z).}
\end{align}
One might wonder what the image of $\cl_{\Omega}$ might be. A first restriction is given by diagram \eqref{cyclediag} itself. For the image of $\cl_{\Omega}$ has to be contained in the subgroup $\Hdg^{2*}_{MU}(X)$ of elements in $MU^{2*}(X)$ which are mapped to Hodge classes under $\vartheta$. But, as Claire Voisin kindly pointed out to us, one can say more. As in Koll\'ar's example, let $X$ be a very general smooth hypersurface in $\Pro^4$ of degree divisible by $s^3$ for an integer $s$ coprime to $6$. Then the map $\vartheta$ is surjective and all integral cohomology classes in $H^4(X; \Z)$ are Hodge classes. Thus, since there is a class $\alpha \in H^4(X; \Z)$ which is not contained in the image of $\cl$, there is an element in $MU^4(X)$ mapping to $\alpha$ which is not contained in the image of $\cl_{\Omega}$. 

\end{remark}

\begin{remark}\label{remtate}
In \cite{cts}, Colliot-Th\'el\`ene and Szamuely use the examples of \cite{ah} to show that the $\ell$-adic integral version of the Tate conjecture for varieties over finite fields of characteristic $\ne \ell$ fails (in \cite{etalecycles}, one can find a different explanation of the obstructions). More recently, it was shown by Pirutka-Yagita \cite{py} for the primes $\ell=2$, $3$ or $5$, and by Kameko \cite{kameko} for all primes $\ell$, that the integral version of the Tate conjecture even fails for non-torsion classes. Those examples are based on the examples by Atiyah-Hirzebruch and approximations of projective varieties by classifying spaces of affine algebraic groups. 
In \cite{antieau}, Antieau shows that there is another type of non-torsion examples which come from classifying spaces of quotients of special linear groups. The argument uses representation theory, while the $Q_i$ vanish on those examples. However, there are differentials of higher degree in the Atiyah-Hirzebruch spectral sequence that do detect the examples. 
\end{remark}

\begin{remark}
By \cite[\S 20]{serre}, Godeaux-Serre varieties can be constructed over any infinite field. In particular, we could use the techniques of \cite{etalecycles} and consider a smooth projective variety $X$ over an algebraically closed field $\bar{\F}_{\ell}$ of characteristic $\ell \ne p$. There is a $p$-completed \'etale version of Brown-Peterson cohomology for $p$ and $n$, denoted by $\hat{BP}\langle n \rangle_{\et}^{*}(X)$. Moreover, the results of \cite{hoyois} allow to construct a motivic version $BP\langle n \rangle_{\Mh}^{2*,*}(X)$ in characteristic $\ell \ne p$. 
The stable $p$-completed \'etale realization functor then induces a natural map 
\begin{align*}
BP\langle n \rangle_{\Mh}^{2*,*}(X)\otimes_{\Z} \Z_{p} \to \hat{BP}\langle n \rangle_{\et}^{2*}(X). 
\end{align*}
Then one can use Milnor operations in \'etale cohomology and repeat our argument for Godeaux-Serre varieties defined over $\bar{\F}_{\ell}$. This yields non-algebraic classes in the \'etale Brown-Peterson tower in characteristic $\ell \ne p$.  
\end{remark}

%
%
%
%
\bibliographystyle{amsalpha}

\end{document}